\documentclass[12pt,a4paper]{amsart}
\usepackage[mathcal]{eucal}

\usepackage{amsmath, amssymb, xy}
\input xy
\xyoption{all}
\usepackage{pdflscape}
\usepackage{hyperref}
\usepackage[vcentermath]{youngtab}
\usepackage[draft]{graphicx}

\numberwithin{equation}{section}

\newtheorem{thm}{Theorem}[section]
\newtheorem{prop}[thm]{Proposition}

\newtheorem{cor}[thm]{Corollary}
\newtheorem{lem}[thm]{Lemma}

\theoremstyle{definition}
\newtheorem{defn}[thm]{Definition}
\newtheorem*{remark*}{Remark}
\newtheorem{remark}[thm]{Remark}

\newtheorem{thmABC}{Fact}

\newtheorem{proABC}[thmABC]{Problem}

\newcommand{\la}{{\lambda}}

\def \gb {{\mathbf g}}
\def \cb {{\mathbf c}}

\newcommand{\cH}{{\mathcal H}}
\newcommand{\cU}{{\mathcal U}}
\newcommand{\cV}{{\mathcal V}}
\newcommand{\cF}{{\mathcal F}}

\newcommand{\cL}{{\mathcal L}}

\newcommand{\cN}{{\mathcal N}}
\newcommand{\cR}{{\mathcal R}}
\newcommand{\cC}{{\mathcal C}}

\newcommand{\cT}{{\mathcal T}}
\newcommand{\cM}{{\mathcal M}}
\newcommand{\bQ}{{\mathbb Q}}
\newcommand{\bZ}{{\mathbb Z}}

\newcommand{\fo}{\mathfrak o}

\newcommand{\GL}{\text{GL}}
\newcommand{\glO}[2]{\text{GL}_{#1}({\mathfrak {o}}_{#2})}

\newcommand{\yab}{Y_{\iota(\la)}}
\newcommand{\pab}{\mathcal{P}_{\iota(\la)}}
\newcommand{\fab}{\mathcal{F}_{ \iota(\la)}}

\newcommand{\hab}{\mathcal{H}_{ \iota(\la) }}
\newcommand{\y}[1]{Y_{#1}}
\newcommand{\f}[1]{\mathcal{F}_{#1}}
\newcommand{\sab}{\mathcal{S}_{ \iota(\la)}}
\newcommand{\kk}{k}

\renewcommand{\gg}{G}
\newcommand{\gh}{\mathrm{GL}_n(\kk)}
\newcommand{\sstd}{\mathrm{sstd}}

\newcommand{\End}{\text{End}}

\newcommand{\Hom}{\text{Hom}}
\newcommand{\Gb}{\bold{G}}

\title[Murphy basis and an application]{Geometric Interpretation of Murphy \\ Bases and an Application}
\date{\today }

\author{Uri Onn} \address{Department of Mathematics, Ben Gurion
  University of the Negev, Beer-Sheva 84105 Israel}
  \email{urionn@math.bgu.ac.il}

\author[Pooja Singla]{Pooja Singla${}^\dag$}\thanks{${}^\dag$Supported by the Center for Advanced Studies in Mathematics at Ben Gurion University} \address{Department of Mathematics, Ben Gurion
  University of the Negev, Beer-Sheva 84105 Israel}
  \email{pooja@math.bgu.ac.il}


\begin{document}

\maketitle
\begin{abstract}
In this article we study the representations of general linear
groups which arise from their action on flag spaces. These
representations can be decomposed into irreducibles by proving
that the associated Hecke algebra is cellular. We give a geometric
interpretation of a cellular basis of such Hecke algebras which
was introduced by Murphy in the case of finite fields. We apply
these results to decompose representations which arise from the
space of submodules of a free module over principal ideal local
rings of length two with a finite residue field.


\end{abstract}

\section{Introduction}
\label{introduction}

\subsection{Flags of vector spaces}\label{field} Let $\kk$ be a finite field and let $n$ be a fixed
positive integer. Let $G=\GL_n(\kk)$ be the group of $n$-by-$n$
invertible matrices over $\kk$ and let $\Lambda_n$ stand for the
set of partitions of $n$. For $\lambda=(\la_i) \in \Lambda_n$,
written in a non-increasing order, let $l(\lambda)$ denote its
length, namely the number of non-zero parts. The set $\Lambda_n$
is a lattice under the opposite dominance partial order, defined
by: $\la \le \mu$ if $\sum_{j=1}^i\la_j \ge\sum_{j=1}^i\mu_j$ for
all $i \in \mathbb{N}$. Let $\vee$ and $\wedge$ denote the
operations of join and meet, respectively, in the lattice
$\Lambda_n$. We call a chain of $\kk$-vector spaces
$\kk^n=x_{l(\lambda)} \supset x_{l(\lambda)-1} \supset \cdots
\supset x_{0} = (0)$ a $\la$-flag if
$\dim_{\kk}(x_{l(\la)-i+1}/x_{l(\la)-i}) = \la_i$ for all $1 \le i
\le l(\la)$. Let
\[
X_\lambda=\{(x_{l(\lambda)-1},\cdots,x_{1}) \mid \kk^n =
x_{l(\lambda)} \supset
 \cdots \supset x_{0} = (0)~\text{is a
 $\lambda$-flag} \},
\]
be the set of all $\la$-flags in $\kk^n$. Let $\cF_\la$ be the
permutation representation of $G$ that arises from its action on
$X_{\la}$ ($\la \in \Lambda_n$). Specifically,
$\cF_{\la}=\bQ(X_\lambda)$ is the vector space of $\bQ$-valued
functions on $X_{\la}$ endowed with the natural $G$-action:
\[
\begin{split}
\rho_{\la}: G& \rightarrow \text{Aut}_{\bQ}(\cF_{\la}) \\
g&  \mapsto   [\rho_{\la}(g)f](x)  = f(g^{-1}x).
\end{split}
\]

Let $\cH_{\la} = \End_G(\cF_{\la})$ be the Hecke algebra
associated to $\cF_\la$. The algebra $\cH_\lambda$ captures the
numbers and multiplicities of the irreducible constituents in
$\cF_\lambda$. The notion of Cellular Algebra, to be described in
Section \ref{preliminaries}, was defined by Graham and Lehrer in
\cite{MR1376244}. Proving that the algebra $\cH_{\la}$ is cellular
gives, in particular, a classification of the irreducible
representations of $\cH_\la$ and hence also gives the
decomposition of $\cF_{\la}$ into irreducible constituents.
Murphy~\cite{MR1194316, MR1327362} gave a beautiful description of
a cellular basis of the Hecke algebras of type $A_n$ denoted
$\cH_{R, q}(S_n)$; cf. \cite{MR1711316}. For $q=|k|$ one has
$\cH_{(1,...,1)} \simeq \cH_{\mathbb Q, q}(S_n)$. Dipper and
James~\cite{MR812444} (see also \cite{MR1711316}) generalized this
basis and constructed cellular bases for the Hecke algebras
$\cH_{\la}$. The first result in this paper is a new construction
of this basis which is of geometric nature. More specifically, the
characteristic functions of the orbits of the diagonal $G$-action
on $X_\la \times X_\mu$ gives a basis of the Hecke modules
$\cN_{\mu\la}=\Hom_G(\cF_\la,\cF_\mu)$. For $\mu \leq \la$ we
allocate a subset of these orbits denoted $\cC_{\mu\la}$ such that
going over all the compositions $\cC^{\mathrm{op}}_{\mu\la} \circ
\cC_{\mu\la}$  and all $\mu \le \la$ gives the desired basis. The
benefit of this description turns out to be an application in the
following setting.

\subsection{Flags of $\mathfrak {o}$-modules}\label{ring}
Let $\mathfrak {o}$ be a complete discrete valuation ring. Let
$\mathfrak{p}$ be the unique maximal ideal of $\mathfrak {o}$ and
$\pi$ be a fixed uniformizer of $\mathfrak{p}$. Assume that the
residue field $ \kk = \mathfrak {o}/\mathfrak{p}$ is finite. The
typical examples of such rings are $\mathbf Z_{p}$ (the ring of
$p$-adic integers) and $\mathbf F_q [[t]]$ (the ring of formal
power series with coefficients over a finite field). We denote by
$\mathfrak {o}_{\ell}$ the reduction of $\mathfrak {o}$ modulo
$\mathfrak{p}^{\ell}$, i.e., $\mathfrak {o}_{\ell} = \mathfrak
{o}/\mathfrak{p}^{\ell}$. Since $\mathfrak {o}$ is a principal
ideal domain with a unique maximal ideal $\mathfrak{p}$, every
finite $\mathfrak {o}$-module is of the form
$\oplus_{i=1}^{j}\mathfrak {o}_{\la_{i}}$, where $\la_{i}$'s can
be arranged so that $\lambda = (\la_1,\ldots,\la_j) \in
\Lambda=\cup \Lambda_{n}$. Let $\glO{n}{\ell}$ denote the group of
$n$-by-$n$ invertible matrices over $\mathfrak {o}_{\ell}$. We are
interested in the following generalization of the discussion in
\S\ref{field}. Let
\[
\cL^{(r)}=\cL^{(r)}({\ell^n})=\{(x_r,\cdots,x_{1}) \mid \mathfrak
{o}_{\ell}^n \supset x_r \supset \cdots \supset x_{0} =
(0),~\text{$x_i$ are $\mathfrak {o}$-modules}\}
\]
be the space of flags of length $r$ of submodules in $\mathfrak
{o}_{\ell}^n$. Let $\Xi \subset \cL^{(r)}$ denote an orbit of the
$\glO{n}{\ell}$-action on $\cL^{(r)}$. Let
$\cF_\Xi=\mathbb{Q}(\Xi)$ be the corresponding permutation
representation of $\glO{n}{\ell}$. One is naturally led to the
following related problems:
\begin{proABC}\label{p1} Decompose $\cF_\Xi$ to irreducible
representations.
\end{proABC}

\begin{proABC}\label{p2} Find a cellular basis for the algebra
$\cH_\Xi=\End_{\glO{n}{\ell}}(\cF_\Xi)$.
\end{proABC}

Few other cases, beside the field case ($\ell=1$) which is our
motivating object, were treated in the literature. The
Grassmannian of free $\mathfrak{o}_\ell$-modules, i.e., the case
$r=1$ and $x_1 \simeq \mathfrak{o}_\ell^m$ is treated fully in
\cite{BO2, MR2283434}. The methods therein are foundational to the
present paper. Another case which at present admits a very partial
solution is the case of complete free flags in
$\mathfrak{o}_\ell^3$; cf. \cite{MR2504482}. In this paper we
treat the first case which is not free but we restrict ourselves
to level $2$, that is, we look at the Grassmannian of arbitrary
$\mathfrak{o}_2$-modules of type ($2^a1^b$) in $\mathfrak{o}_2^n$.
To solve this problem we are naturally led to consider certain
spaces of $2$-flags of $\mathfrak{o}_2$-modules as well. We give a
complete solution to problems \ref{p1} and \ref{p2} in these
cases.


\section{Preliminaries}\label{preliminaries}


\subsection{Hecke algebras and Hecke modules}
For $\la,\mu \in \Lambda_n$, we let
$\cN_{\la\mu}=\Hom_G(\cF_\mu,\cF_\la)$ denote the
$\cH_\la$-$\cH_\mu$-bimodule of intertwining $G$-maps from
$\cF_\mu$ to $\cF_\la$. The modules $\cN_{\la\mu}$, and in
particular the algebras $\cH_\lambda$, have natural \lq geometric
basis\rq~ indexed by $X_\lambda \times_G X_\mu$, the space of
$G$-orbits in $X_\lambda \times X_\mu$ with respect to the
diagonal $G$-action. Specifically, for $\Omega \in X_{\la}
\times_G X_{\mu}$, let
 \begin{equation}\label{geometric.basis}
 \gb_\Omega f (x)= \sum_{y:(x,y) \in \Omega} f(y), \qquad f \in
 \cF_\mu,\, x \in X_\la.
\end{equation}
Then $\{\gb_\Omega \mid \Omega \in X_\la \times _G X_\mu \}$ is a
basis of $\cN_{\la\mu}$. Let $\cM_{\la \mu}$ be the set of
$l(\la)$-by-$l(\mu)$ matrices having non-negative integer entries
with column sum equal to $\mu$ and row sum equal to $\la$, namely
\begin{equation}\label{intersection-matrix}
\cM_{\la \mu} = \{(a_{ij}) \mid a_{ij} \in \bZ_{\geq
0},\,\sum_{i=1}^{l(\la)} a_{ij} = \mu_j, \sum_{j=1}^{l(\mu)}
a_{ij} = \la_i \}.
\end{equation}
Geometrically, the orbits in $X_{\la} \times_G X_{\mu}$
characterize the relative positions of $\la$-flags and $\mu$-flags
in $\kk^n$ and hence are in bijective correspondence with the set
$\cM_{\la\mu}$. The bijection
\begin{equation}\label{orbits-matrices}
X_{\la} \times_G X_{\mu} \longleftrightarrow  \cM_{\la\mu},
\end{equation}
is obtained by mapping the pair $(x, y) \in X_{\la} \times X_{\mu}$ to
its intersection matrix $(a_{ij}) \in \cM_{\la\mu}$, defined by
\begin{equation}
\label{intersection matrix}
a_{ij} = \dim_{\kk}\left( \frac{x_i \cap y_j}{x_i \cap y_{j-1} + x_{i-1} \cap y_j}\right).
\end{equation}

\subsection{The RSK Correspondence}\label{subsec:Young} A Young diagram of a partition $\mu \in \Lambda_n$ is the set
$[\mu] = \{(i,j) \mid 1 \leq j \leq \mu_i \, \mathrm{and} \, 1 \le
i \le l(\mu) \} \subset \mathbb N \times \mathbb N$. One usually
represent it by an array of boxes in the plane, e.g.\ if $\mu =
(3,2)$ then $[\mu] = {\tiny \yng(3,2)}$. A $\mu$-tableau $\Theta$
is a labeling of the boxes of $[\mu]$ by natural numbers. The
partition $\mu$ is called the shape of $\Theta$ and denoted
$\mathrm{Shape}(\Theta)$. A Young tableau is called {\it
semistandard} if its entries are increasing in rows from left to
right and are strictly increasing in columns from top to bottom. A
semistandard tableau of shape $\mu$ with $\sum \mu_i = n$ is
called {\it standard} if its entries are integers from the set
$\{1, 2, \ldots, n\}$, each appearing exactly once and strictly
increasing from left to right as well. Given partitions $\nu$ and
$\mu$, a tableau $\Theta$ is called of \lq shape $\nu$ and type
$\mu$\rq~ if it is of shape $\nu$ and each natural number $i$
occurs exactly $\mu_i$ times in its labeling. We denote by
$\mathrm{std}(\nu)$ and $\sstd(\nu\mu)$ the set of all standard
$\nu$-tableaux and set of semistandard $\nu$-tableaux of type
$\mu$, respectively. We remark that the set $\sstd(\nu\mu)$ is
nonempty if and only if $\nu \le \mu$.

\smallskip
The RSK correspondence is an algorithm which explicitly defines a
bijection
\[
\cM_{\la \mu} \longleftrightarrow \bigsqcup_{\nu \le \la \wedge
\mu } \sstd(\nu \la) \times \sstd(\nu \mu),
\]
where $\la \wedge \mu$ is the meet of $\la$ and $\mu$. For more
details on this see \cite{MR0272654}.

\begin{defn}
\label{Embedding of flags} We say that a $\nu$-flag $y$ is {\em
embedded} in a $\mu$-flag $x$, denoted $y \hookrightarrow x$, if
$l(\nu) \leq l(\mu)$ and $y_{l(\nu)-i} \subset x_{l(\mu)-i}$ for
all $1\leq i \leq l(\nu)$. The intersection matrix of each
embedding of $\nu$-flag into $\mu$-flag determines a $\nu$-tableau
of type $\mu$ as follows: for any intersection matrix $E =
(a_{ij}) \in \cM_{\nu\la}$, construct the Young tableau with
$a_{ij}$ many $i$'s in its $j^{\mathrm{th}}$ row. We call an
embedding of $\nu$-flag into a $\mu$-flag {\em permissible} if the
 $\nu$-tableau obtained is semistandard. The set
$\cM_{\nu\la}^{\circ}$ denotes the subset of $\cM_{\nu\la}$
consisting of intersection matrices that corresponds to
permissible embeddings.
\end{defn}

The following gives a reformulation of the RSK correspondence purely in
terms of intersection matrices:
\begin{equation}\label{the.RSK.in.terms.of.matrices}
\cM_{\la\mu} \longleftrightarrow \bigsqcup_{\nu \le \la \wedge
\mu} \cM_{\nu \la}^\circ \times \cM_{\nu \mu}^\circ.
\end{equation}

For partitions $\nu \le \la$ of $n$ we let $(X_\nu \times
X_\la)^\circ$ denote the subset of $X_\nu \times X_\la$ which
consists of pairs $(z,x)$ such that $z$ is permissibly embedded in
$x$. The orbits $(X_\nu \times_G X_\la)^\circ$ are therefore
parameterized by $\cM_{\nu \la}^\circ$. This gives a purely
geometric reformulation of the RSK correspondence:
\begin{equation}\label{geometric.RSK}
X_{\la} \times_G X_{\mu} \tilde{\longrightarrow}
\bigsqcup_{\nu \le \la \wedge \mu} \left(X_\nu \times_G
X_\la\right)^\circ \times \left(X_\nu \times_G X_\mu\right)^\circ.
\end{equation}
The gist of \eqref{geometric.RSK} is that both sides have
geometric interpretations. We remark that the above bijection is
an important reason behind the cellularity of MDJ basis (see
Section~\ref{M-DJ bases}).


\subsection{Cellular Algebras}

\label{Cellularity} Cellular algebras were defined by Graham and
Lehrer in \cite{MR2283434}. We use the following equivalent
formulation from Mathas~\cite{MR1711316}.

\begin{defn}
\label{cellular algebra} Let $K$ be a field and let $A$ be an
associative unital $K$-algebra. Suppose that $(\zeta, \geq)$ is a
finite poset and that for each $\tau \in \zeta$ there exists a
finite set $\mathcal{T}(\tau)$ and elements $c_{st}^{\tau} \in A$
for all $s, t \in \mathcal{T}(\tau)$ such that $\mathcal{C} = \{
c_{st}^{\tau} \mid \tau \in \zeta \,\,\text{and}\,\,s, t \in
\mathcal{T}(\tau) \}$ is a basis of $A$. For each $\tau \in \zeta$
let $\tilde{A}^{\tau}=\mathrm{Span}_K \{c_{uv}^{\omega} \mid
\omega \in \zeta,\,\, \omega > \tau \,\,\text{and} \,\,u, v \in
\mathcal{T}(\omega) \}$. The pair $(\mathcal{C}, \zeta)$ is a
cellular basis of $A$ if
\begin{enumerate}
\item The $K$-linear map $\star : A \rightarrow A$ determined by $c_{st}^{\tau \star} = c_{ts}^{\tau}$ ($\tau \in \zeta, s,t \in \mathcal{T}(\tau)$)
is an algebra anti-homomorphism of $A$; and,
\item for any $\tau \in \zeta$, $t \in \mathcal{T}(\tau)$ and $a \in A$ there exists $\{\alpha_{v} \in K \mid v \in \cT(\tau)\}$ such that for all
$s \in \mathcal{T}(\tau)$
\begin{equation}
\label{cellular condition}
a \cdot c_{st}^{\tau} = \sum_{v \in \cT(\tau)} \alpha_{v}c_{vt}^{\tau}
\,\,\text{mod}\, \tilde{A}^{\tau}.
\end{equation}
\end{enumerate}
If $A$ has a cellular basis then $A$ is called a cellular algebra.
\end{defn}

The result about semisimple cellular algebras which we shall need
is the following. Let $A$ be a semisimple cellular algebra with a
fixed cellular basis $(\mathcal C = \{c_{st}^{\tau} \}, \zeta)$.
For $\tau \in \zeta$ let $A^{\tau}$ be the $K$-vector space with
basis $\{ c_{uv}^{\mu} \mid \mu \in \zeta, \mu \geq \tau
~\mathrm{and}~ u,v \in \mathcal T(\mu)\}$. Thus $\tilde A^{\tau}
\subset A^{\tau}$ and $A^{\tau} / \tilde A^{\tau}$ has basis
$c_{st}^{\tau} + \tilde A^{\tau}$ where $s, t \in
\mathcal{T}(\tau)$. It is easy to prove that $A^{\tau}$ and
$\tilde A^{\tau}$ are two sided ideals of $A$. Further if $s,t \in
\cT(\tau)$, then there exists an element $\alpha_{st} \in K$ such that
for any $u, v \in \cT(\tau)$
\[
c_{us}^{\tau} c_{tv}^{\tau} = \alpha_{st} c_{uv}^{\tau} \,\,
\mathrm{mod} \,\, \tilde A^{\tau}.
\]
For each $\tau \in \zeta$ the cell modules $C^{\tau}$ is defined
as the left $A$ module with $K$-basis $\{ b_{t}^{\tau} \mid t \in
\mathcal {T}(\tau) \}$ and with the left $A$ action:
\[
a \cdot b_{t}^{\tau} = \sum_{v \in \mathcal{T}(\la) } \alpha_{v}
b_{v}^{\tau},
\]
for all $a \in A$ and $\alpha_{v}$ are as given in the
Definition~\ref{cellular algebra}. Furthermore, dual to $C^{\tau}$
there exists a right $A$-modules $C^{\tau *}$ which has the same
dimension over $K$ as $C^{\tau}$, such that the $A$-modules
$C^{\tau } \otimes_{K} C^{\tau *}$ and $A^{\tau}/\tilde{A}^{\tau}$
are canonically isomorphic.
\begin{thm} \cite[Lemma~2.2 and Theorem~3.8]{MR2283434} Suppose $\zeta$ is finite. Then $\{ C^{\tau} \, | \, \tau \in \zeta \, \mathrm{and} \, C^{\tau} \neq 0 \}$ is a complete set of pairwise
inequivalent irreducible $A$-modules. Let $\zeta^{+}$ be the set
of elements $\tau \in \zeta$ such that $C^{\tau} \neq 0$. Then $A
\cong \oplus_{\tau \in \zeta^{+}} C^{\tau } \otimes_{K} C^{\tau
* }$.
\end{thm}


\section{Another description of Murphy-Dipper-James bases}


\subsection{MDJ Bases}
\label{M-DJ bases} For a positive integer $n$, let $S_n$ be the
symmetric group of $\{ 1,2, \ldots, n\}$. Let $S$ be the subset of
$S_n$ consisting of the transpositions $(i, i+1)$. Let $R$ be a
commutative integral domain and let $q$ be an arbitrary element of
$R$. The Iwahori-Hecke algebra $\mathcal H_{R, q}(S_n)$ is the
free $R$-module generated by $\{ T_{\omega} \mid \omega \in S_n
\}$ with multiplication given by
\[T_{w} T_s =
  \begin{cases}
    T_{ws} & \text{if} \;\; \ell(ws) > \ell(w), \\
    qT_{ws} + (q-1)T_w & \text{otherwise},
  \end{cases}
\]
where $\ell(w)$ denotes the length of $w \in S_n$. Also $\star: \mathcal H_{R, q}(S_n) \rightarrow \mathcal H_{R, q}(S_n)$ denotes an algebra anti-involution defined by $T^{\star}_{\omega} = T_{\omega^{-1}}$. For a partition $\mu$, let $S_{\mu}$ be the subset of $S_n$ consisting of all permutations leaving the sets $\{\sum_{i=1}^{j-1} \mu_i +1,....,\sum_{i=1}^{j} \mu_i \}$ invariant for all $1 \leq j \leq l(\mu)$ and let $m_{\mu} = \sum_{\omega \in S_{\mu}} T_{\omega}$. Let $\cN_{\la\mu}^q$ denote the free $R$-module $m_{\la} \cH_{R, q}(S_n) m_{\mu}$.  \\

For each partition $\nu$ of $n$, let $\phi^{\nu}$ be the unique $\nu$-standard tableau
in which the integers $\{1, 2, \ldots, n\}$ are entered in increasing order from left to right
along the rows of $[\nu]$. For each $\nu$-standard tableau $\theta$ define the permutation
matrix $d(\theta)$ by $\theta = \phi^{\nu} d(\theta)$.
For any standard $\nu$-tableau $\theta$ and partition $\mu$ of $n$ such that $\nu \leq \mu$, we obtain a semistandard $\nu$-tableau of type $\mu$, denoted $\mu(\theta)$, by replacing each entry $i$ in $\theta$ by $r$ if $i$ appears in row $r$ of $\phi^{\mu}$.
For given partitions $\mu$ and $\nu$, let $\Theta_1 \in \mathrm{sstd}(\nu, \la)$ and $\Theta_2 \in \mathrm{sstd}(\nu, \mu)$, define
\[
m_{\Theta_1 \Theta_2} = \sum_{\theta_1, \theta_2} m_{\theta_1 \theta_2},
\]
where $m_{\theta_1 \theta_2} = T^{\star}_{d(\theta_1)} m_{\nu}
T_{d(\theta_2)}$ and the sum is over all pairs
$(\theta_1,\theta_2)$ of standard $\nu$-tableau such that
$\la(\theta_1) = \Theta_1$ and $\mu(\theta_2) = \Theta_2$. Let
\[\mathbb M_{\mu\la}  = \{m_{\Theta_1 \Theta_2} \mid \Theta_1 \in
\mathrm{sstd}(\nu, \la), \Theta_2 \in \mathrm{sstd}(\nu, \mu), \nu
\leq \la \wedge \mu  \},\] and for any partition $\mu \in
\Lambda_n$, let $\Lambda_{\mu}=\{ \nu \in \Lambda_n \mid \nu \le
\mu\}$. Then
\begin{thm}[Murphy, Dipper-James]
\label{thm:M-DJ} The set $(\mathbb M_{\mu\la}, \Lambda_{\la \wedge
\mu})$ is an $R$-basis of the Hecke module $\cN^q_{\mu\la}$.
\end{thm}
\begin{proof} See Mathas~\cite[Theorem~4.10, Corollary~4.12]{MR1711316}.
\end{proof}
\begin{remark}
\label{rk:M-DJ} The following observation from the proof is important for us.
For any semistandard $\nu$-tableau $\Theta$, let $\mathrm{first} (\Theta)$ be the unique row standard $\nu$-tableau such that $\la(\mathrm{first}(\Theta)) = \Theta$.
For $\Theta \in \mathrm{sstd}(\nu, \la)$,
\begin{equation}
\label{remark from mathas}
\Gb_{\Theta} := \sum_{\theta \in \mathrm{std}(\nu), \atop \la(\theta) = \Theta} m_{\nu}^{\star} T_{d(\theta)} = \sum_{\omega \in S_{\nu} \sigma S_{\la}} T_{\omega},
\end{equation}
where $\sigma \in S_n$ is the unique permutation matrix satisfying $\sigma = d(\mathrm{first}(\Theta))$ (see also the Remark~\ref{echelon form}).
\end{remark}
Any partition $\delta = (\delta_i)$ associates $l(\delta)$ many $\delta$-row ($\delta$-column) submatrices with a given $n \times n$ matrix $A$ by taking its rows (columns) from $\sum_{i=0}^j \la_i +1$ to $\sum_{i=0}^j \la_{i+1}$ for all $0 \leq j \leq l(\delta)-1$.
\begin{defn}($\la\mu$-Echelon form)
A matrix $A$ is called in $\la \mu $-Echelon form if its associated $\la$-row and $\mu$-column sub-matrices are in row reduced and column reduced Echelon form respectively.
\end{defn}
\begin{remark}
\label{echelon form}
The matrices $\sigma_1$ and $\sigma_2$ appearing in the proof of Theorem~\ref{thm:M-DJ} and Remark~\ref{rk:M-DJ} are in $\la\nu$ and $\mu\nu$
Echelon form respectively.
\end{remark}

\subsection{Geometric interpretation of the MDJ Bases} Recall RSK correspondence is an algorithm that explicitly defines the correspondence:
\[
\xymatrix{
\cM_{\la\mu}   \ar@{<->}[r] \ar@{<->}[d] & \bigsqcup_{\nu \in \Lambda_{\la \wedge \mu}} (X_{\nu} \times_G X_{\la})^{\circ} \times (X_{\nu} \times_G X_{\mu})^{\circ} \ar@{<->}[d] \\
\bigsqcup_{\nu \in \Lambda_{\la \wedge \mu}} \cM_{\nu\la}^{\circ} \times \cM_{\nu\mu}^{\circ} \ar@{<->}[r] & \bigsqcup_{\nu \in \Lambda_{\la \wedge \mu}} \mathrm{sstd}(\nu, \la) \times \mathrm{sstd}(\nu, \mu), }
\]
where the upper left corner consists of intersection matrices, the
upper right consists of orbits of permissible embeddings, the
lower left corner consists of intersection matrices of permissible
embedding and the lower right of semistandard tableaux. For $\nu
\in \Lambda_{\la \wedge \mu}$ and orbits $ \Omega_1 \in
\left(X_\nu \times_G X_\la\right)^\circ$, $\Omega_2 \in
\left(X_{\nu} \times_G X_{\mu} \right)^{\circ}$ define
\[
\cb^\nu_{\Omega_1\Omega_2}:=\gb_{\Omega_1^{\text{op}}} \circ
\gb_{\Omega_2},
\]
where $[(x,y)]^{\text{op}}=[(y,x)]$. Clearly,
$\cb^\nu_{\Omega_1\Omega_2} \in \cN_{\la\mu}$. For any partition
$\nu$, let $P_{\nu}$ be the stabilizer of standard $\nu$-flag in
$G$ and $B$ be the Borel subgroup of $G$, that is the subgroup
consisting of all invertible upper triangular matrices. Let
$\cC_{\mu \la} = \{\cb_{\Omega_1\Omega_2}^\nu \mid \nu \in
\Lambda_{\la \wedge \mu},\,\,\Omega_1 \in \left(X_\nu \times_G
X_\la\right)^\circ, \Omega_2 \in (X_{\nu} \times_G
X_{\mu})^{\circ} \}$.
\begin{thm}
\label{thm:Murphy=our}
The set $(\cC_{\mu\la}, \Lambda_{\la \wedge \mu})$ is a $\mathbb Q$-basis of $\cN_{\la\mu}$.
\end{thm}
\begin{proof} We prove this by proving that
if $q$ is the cardinality of field $\kk$ then the set
$(\cC_{\mu\la}, \Lambda_{\la \wedge \mu})$ coincides with MDJ
basis of $\cN_{\la\mu}^q$ up to a scalar. That is, if the orbits
$\Omega_1$ and $\Omega_2$ correspond to semistandard tableau
$\Theta_1$ and $\Theta_2$ by RSK, then
\[
\cb_{\Omega_1 \Omega_2}^{\nu} = \frac{|P_{\nu}|}{|B|} m_{\Theta_1
\Theta_2}.
\]
By definition of $m_{\Theta_1\Theta_2}$ and (\ref{remark from mathas}),
\begin{eqnarray}
m_{\Theta_1 \Theta_2}  =  \sum_{\theta_1, \theta_2} m_{\theta_1 \theta_2} = \sum_{\theta_1, \theta_2} T^{\star}_{d(\theta_1)} m_{\nu} T_{d(\theta_2)}.
\end{eqnarray}
By using the observations:
\begin{enumerate}
\item $m_{\nu}^{\star} = m_{\nu}$, $m_{\nu}^{2} = \sum_{w \in S_{\nu}} q^{l(w)} m_{\nu}$.
\item $\sum_{w \in S_{\nu}} q^{l(w)} = \frac{|P_{\nu}|}{|B|}$.
\end{enumerate}
We obtain
\[
m_{\Theta_1 \Theta_2} = \frac{|B|}{|P_{\nu}|} \Gb_{\Theta_1}^{\star} \Gb_{\Theta_2}.
\]
We claim that if the semistandard tableau $\Theta_i$ corresponds
to the orbit $\Omega_i$ by RSK then $\Gb_{\Theta_i} =
\gb_{\Omega_i}$. We argue for $i=1$. We have an isomorphism
\[
 \mathcal {H}_{\mathbb Q, q}(S_n) \cong \mathbb{Q} [B\backslash G / B],
\]
such that a basis element $T_{\omega}$ in the Hecke algebra
$\cH_{\bQ, q}(S_n)$ corresponds to the function $\bold{1}_{BwB}
\in \mathbb Q[B\backslash G /B]$. The commutativity of the
following diagram implies that the sum $\sum_{w \in S_\nu \sigma_1
S_{\la}} T_w$ corresponds to $\bold{1}_{P_{\nu} \sigma_1
P_{\la}}$.
\begin{eqnarray*}
 B\backslash G / B & \leftrightarrow & S_{n} \\
\downarrow       &                 &   \downarrow    \\
P_{\nu} \backslash G / P_{\la} & \leftrightarrow & S_{\nu} \backslash S_n /S_{\la}
\end{eqnarray*}
Therefore $\Gb_{\Theta_1}$ belongs to $\mathrm{Hom}_G(\mathcal
F_{\la}, \mathcal F_{\nu})$. Let orbit $\Omega_1$ corresponds to
matrix $m = (m_{ij}) \in \cM_{\nu\la}^{\circ}$ by RSK.  Then by
its definition, matrix $\sigma_1$ is the unique matrix in
$\la\mu$-Echelon form such that when viewed as a block matrix
having $(i,j)^{\mathrm{th}}$ block of size $\la_i \times \nu_j$
for $1 \leq i \leq l(\la)$ and $1 \leq j \leq l(\nu)$ then
$m_{ij}$ = sum of entries of $(i,j)^{\mathrm{th}}$ block of
$\sigma_1$. This implies that if $(y,x) \in \Omega_1$, then there
exist full flags $\bar{y}$ and $\bar{x}$ that extend the flags $y$
and $x$ respectively such that the intersection matrix of
$\bar{y}$ and $\bar{x}$ is $\sigma_1$. Therefore,
\[
\gb_{\Omega_1} = \bold{1}_{P_{\nu} \sigma_1 P_{\la}} = \Gb_{\Theta_1}.
\]
Further, the anti-automorphism $\star$ on Hecke algebra $\cH_{\bQ,
q}(S_n)$ coincides with 'op' on $X_{\la} \times_G X_{\la}$, hence
the result.
\end{proof}
For $\mu \in \Lambda$, let $\cH_{\la}^{\mu} = \mathrm{Span}_{\bQ}
\{ \cb_{\Omega_1 \Omega_2}^{\nu} \mid \nu \leq \mu \} $ and
$\cH_{\la}^{\mu-} = \mathrm{Span}_{\bQ} \{ \cb_{\Omega_1
\Omega_2}^{\nu} \mid \nu < \mu \} $.
\begin{prop}\label{ideal structure} \qquad

\begin{enumerate}
\item [(a)] Let $f \in \mathrm{Hom}_G(\cF_{\la}, \cF_{\mu})$ and $h \in \mathrm{Hom}_G(\cF_{\mu}, \cF_{\la})$ then $h\circ f \in \cH_{\la}^{\mu}$.
\item [(b)] The spaces $\cH_{\la}^{\mu}$ and $\cH_{\la}^{\mu-}$ are two sided ideal of $\cH_{\la}$.
\end{enumerate}
\end{prop}
\begin{proof} (a) We prove this by induction on the partially ordered set $\Lambda_{\la}$. If $\delta = (n)$, the partition with only one
part, then $\cF_{\delta}$ is the trivial representation, and it is
easily seen that $\cN_{\la\delta}$ and $\cN_{\delta\la}$ are one
dimensional. It follows that $\cH_{\la}^{\delta}=\cN_{\la \delta}
\circ \cN_{\delta \la}$ is one dimensional and spanned by
$\cb_{[(x,0)],[(0,x)]}^\delta$. This established the basis for the
induction. Now assume the assertion is true for any $\nu \in
\lambda_{\la}$ such that $\nu < \mu$. We prove the result for
$\mu$. Let $p_i$ for $1 \leq i \leq r$ be all the permissible
embeddings of $\mu$-flags into $\la$-flags and let $\Omega_{p_i}
\in (X_{\mu} \times_G X_{\la})^{\circ}$ be the orbits
corresponding to these embeddings. The orbit corresponding to the
identity mapping of $\mu$-flag into itself is denoted by
$\Omega_{\mathrm{id}}$. Let ${\cN}'_{\la \mu}$ be the subspace of
$\cN_{\la \mu}$ generated by the set $\{ \cb_{\Omega_{1}
\Omega_2}^{\nu}  \mid \nu < \mu, \Omega_1 \in (X_{\nu} \times_G
X_{\la})^{\circ}, \Omega_2 \in (X_{\nu} \times_G
X_{\mu})^{\circ}\}$. It follows that any $h \in
\mathrm{Hom}_G(\cF_{\mu}, \cF_{\la})$ can be written as linear
combination $\sum_{i=1}^r \alpha_i \cb_{\Omega_{\mathrm{id}}
\Omega_{p_i}}^{\mu} \;\; \mathrm{mod}\,\, {\cN'_{\la \mu}}$.
Therefore, it is enough to prove that $ \cb_{\Omega_{p_i}
\Omega_{\mathrm{id}}}^{\mu} f \in \cH_{\la}^{\mu}$ for all $1 \leq
i \leq r$. Arguing similarly for $f$ it suffices to prove that
$\cb_{\Omega_{p_{i}} \Omega_{id}}^{\mu} \cb_{\Omega_{\mathrm{id}}
\Omega_{p_j}}^{\mu} \in \cH_{\la}^{\mu}$ for all $1 \leq i,j \leq
r$. But since $\cb_{\Omega_{p_{i}} \Omega_{\mathrm{id}}}^{\mu}
\cb_{\Omega_{\mathrm{id}} \Omega_{p_j}}^{\mu} =
\cb_{\Omega_{p_{i}} \Omega_{p_j}}^{\mu}$, the result follows.

\noindent (b) The fact that $\cH_{\la}^{\mu}$  is a two-sided
ideal follows immediately from (a) and the fact that
$\{\cb_{\Omega_2 \Omega_1}^{\mu}\mid \mu \le \la, \Omega_1,
\Omega_2\}$ is a basis by observing that compositions of basis
elements of the form $\cb_{\Omega'_1\Omega'_2}^{\nu} \cb_{\Omega_1
\Omega_2}^{\mu}$
\[
\xymatrix{
\cF_{\la} \ar[dr]|{{\gb_{\Omega_2}}} &  & \cF_{\la} \ar[dr]|{{\gb_{\Omega_2'}}} &  & \cF_{\la} \\
  &             \cF_{\mu}  \ar[ur]|{{\gb_{\Omega_1^{\mathrm{op}}}}}  &  &  \cF_{\nu} \ar[ur]|{{\gb_{{\Omega_1'}^{\mathrm{op}}}}} &   }
\]
lies in $\cH_\la^{\mu \wedge \nu}$. Finally, as
$\cH_{\la}^{\mu-}=\sum_{\nu < \mu}\cH_{\la}^{\nu}$, the latter is
an ideal as well.

\end{proof}

\begin{thm}
\label{thm: cellular} The Hecke algebra $\cH_{\la}$ is cellular
with respect to $(\cC_{\la\la}, \Lambda_{\la} )$.
\end{thm}
\begin{proof} We have a natural anti-automorphism of the Hecke algebras $\cH_{\la}$ defined as
\[
(\cb_{\Omega_1 \Omega_2}^{\mu})^{\star} = \cb_{\Omega_2 \Omega_1}^{\mu}.
\]
Proposition \ref{ideal structure} implies that the criterion
\ref{cellular condition} for cellularity is fulfilled as well.
\end{proof}

\begin{cor}
\label{main} There exists a collection $\{\cU_{\la} \mid \la \in
\Lambda_n\}$ of inequivalent irreducible representations of
$\mathrm{GL}_n(\kk)$ such that
\begin{enumerate}
\item $\cF_{\la} = \oplus_{\nu \leq \la} \cU_{\nu}^{|\cM_{\nu
\la}^\circ|}$;
\item For every $\mu, \nu \leq \la$, one has $\dim_\bQ \mathrm{Hom}_G \left( \cU_{\nu}, \cF_{\mu}\right) = |\cM_{\nu\mu}^\circ|$. That is,
the multiplicity of $\cU_{\nu}$ in $\cF_{\mu}$ is the number of
non-equivalent permissible embeddings of a $\nu$-flag in a
$\mu$-flag. In particular $\cU_{\nu}$ appears in $\cF_{\nu}$ with
multiplicity one and does not appear in $\cF_{\mu}$ unless $\nu
\leq \mu$.
\end{enumerate}
\end{cor}

We remark that part (2) of Theorem \ref{main} gives a
characterization of the irreducible representations $\cU_\lambda$,
that is, for each $\lambda \in \Lambda_n$, the representation
$\cU_\lambda$ is the unique irreducible representation which
occurs in $\cF_\lambda$ and do not occur in $\cF_\mu$ for $\mu
\leq \lambda$. \\

\subsection{General flags}
\label{subsec: general flags}In this section we extend our results
of previous section to the flags not necessarily associated with
partitions. A tuple $c= (c_i)$ of positive integers such that
$\sum c_{i} = n$ is called composition of $n$. The length of $c$,
denoted $l(c)$ is the the number of its nonzero parts. By
reordering parts of a composition in a decreasing order we obtain
the unique partition associated with it. We shall use bar to
denote the associated partition. For example if $c =(2,1,2)$, then
$\bar{c} = (2,2,1)$. A chain of $\kk$-vector spaces $ x = (\kk^n =
x_{l(c)} \supset x_{l(c)-1} \supset \cdots \supset x_1 \supset
x_{0} = (0))$ is a $c$-flag if
$\dim_{\kk}(x_{l(c)-i+1}/x_{l(c)-i}) = c_i$ for all $1 \le i \le
l(c)$. Let $X_c$ be the space of all $c$-flags and $\cF_{c} =
\mathbb Q(X_c)$. By the theory of representation of symmetric
groups and Bruhat decomposition, it follows that for any
compositions $c_1$ and $c_2$, the Hecke algebras $\cH_{c_1} =
\mathrm{Hom}_G(\cF_{c_1}, \cF_{c_2})$ and $ \cH_{\bar{c}} =
\mathrm{Hom}_G(\cF_{\bar{c_1}}, \cF_{\bar{c_2}})$ are isomorphic
as $G$-modules. By composing this isomorphism with the cellular
basis of the Hecke algebra $\cH_{\bar{c}}$, one obtains the
cellular basis of the Hecke algebras $\cH_{c}$. This implies that
irreducible components of $\cF_{c}$ are parameterized by the set
of partitions $\la \in \Lambda$ such that $\la \leq \bar{c}$. In
particular this gives the following bijection
\[
X_{c_1} \times_G X_{c_2} \longleftrightarrow \bigsqcup_{\nu \in \Lambda, \nu \leq \bar{c}} (X_{\nu} \times_G X_{c_1})^{\circ} \times (X_{\nu} \times_G X_{c_2})^{\circ}
\]
for certain subsets $(X_{\nu} \times_G X_{c_1})^{\circ}$ and
$(X_{\nu} \times_G X_{c_2})^{\circ}$ of $X_{\nu} \times_G X_{c_1}$
and $X_{\nu} \times_G X_{c_2}$ respectively. For any $(x,y) \in
(X_{\nu} \times_G X_{c_1})^{\circ}$, we say $x$ has permissible
embedding in $y$. Whenever we deal with compositions in later section, by cellular basis and permissible embedding we shall mean the general notions defined in this section.

\newpage


\section{The Module Case}


In this section $\mathfrak {o}$ denotes a complete discrete
valuation ring with maximal ideal $\mathfrak{p}$ and fixed
uniformizer $\pi$. Assume that the residue field $ \kk = \mathfrak
{o}/\mathfrak{p}$ is finite. We denote by $\mathfrak {o}_{\ell}$
the reduction of $\mathfrak {o}$ modulo $\mathfrak{p}^{\ell}$,
i.e., $\mathfrak {o}_{\ell} = \mathfrak {o}/\mathfrak{p}^{\ell}$.
Since $\mathfrak {o}$ is a principal ideal domain with a unique
maximal ideal $\mathfrak{p}$, every finite $\mathfrak {o}$-module
is of the form $\oplus_{i=1}^{j}\mathfrak {o}_{\la_{i}}$, where
the $\la_{i}$'s can be arranged so that $\lambda = (\la_i) \in
\Lambda=\cup \Lambda_n$. The rank of an $\mathfrak {o}$-module is
defined to be the length of the associated partition. Note that in
this section we use arbitrary partitions rather than partitions of
a fixed integer and parameterize different objects than in the
previous sections: types of $\mathfrak{o}$-modules rather than
types of flags of $\kk$-vector spaces. Let $\tau$ be the type map
which maps each $\mathfrak {o}$-module to its associated
partition. The group $\glO{n}{\ell}$ denotes the set of invertible
matrices of order $n$ over the ring $\mathfrak {o}_{\ell}$. Let
\[
\cL^{(r)}=\cL^{(r)}({\ell^n})=\{(x_r,\cdots,x_{1}) \mid \mathfrak
{o}_{\ell}^n \supset x_r \supset \cdots \supset x_{0} =
(0),~\text{$x_i$ are $\mathfrak {o}$-modules}\}
\]
be the space of flags of modules on length $r$ in $\mathfrak
{o}_{\ell}^n$.
There is a natural partial ordering on $\cL^{(r)}$ defined by
$\eta=(y_r,...,y_1) \le (x_r,...,x_1)=\xi$ if there exist
embeddings $\phi_1,...,\phi_r$ such that the diagram
\[
\begin{matrix}
x_r & \supset & x_{r-1} & \supset & \cdots & \supset & x_1  \\
\uparrow_{\phi_r} &  &  \uparrow_{\phi_{r-1}} & & \cdots & &  \uparrow_{\phi_1} \\
y_r & \supset & y_{r-1} & \supset & \cdots & \supset & y_1
\end{matrix}
\]
is commutative. Two flags $\xi$ and $\eta$ are called equivalent,
denoted $\xi \sim \eta$, if the $\phi_i$'s in the diagram are
isomorphisms. For any equivalence class $\Xi=[\xi]$ let
$\cF_\Xi=\mathbb{Q}(\Xi)$ denote the space of rational valued
functions on $\Xi$ endowed with the natural
$\glO{n}{\ell}$-action. We use the letter $\Xi$ to denote a set of
flags as well as the {\em type} of the flags in this set.

\smallskip

En route to developing the language and tools for decomposing the
representations $\cF_\Xi$ into irreducible representations we
treat here the special case $\ell=2$ and give a complete spectral
decomposition for the $\glO{n}{2}$-representations $\cF_\Xi$ with
$\Xi \subset \cL^{(1)}({2^n})$. Recall $ \Xi \in \cL^{(1)}(2^n)$
consists of all submodules $x \subset \mathfrak{o}_2^n$ with a
fixed type $\la$. We shall also assume that $n \geq
2(\mathrm{Rank}(x))$ and $\mathrm{Rank}(x) \geq
2(\mathrm{Rank}(\pi x))$. We have a map $\iota: \cL^{(1)} \to
\cL^{(2)}$ given by $y \mapsto (y,\pi y)$ which allows us to
identify any module with a (canonically defined) pair of modules.
We will see that to find and separate the irreducible constituents
of $\cF_{\la}$ with $\la \subset \cL^{(1)}$ we need to use a
specific set of representations $\cF_\eta$ with $\eta \in
\cL^{(2)}$ such that $\eta \le \iota(\la)$. A similar phenomenon
has been observed also in \cite{MR2283434}. We remark that for the
groups $\glO{n}{2}$, it is known that the dimensions of complex
irreducible representations and their numbers in each dimension
depend only on the cardinality of residue field of $\mathfrak
{o}$, see \cite{MR2684153}. For the current setting we shall prove
that the numbers and multiplicities of the irreducible
constituents of $\cF_\lambda$ with $\la \subset \cL^{(1)}$ are
independent of the residue field as well, though this is not true
in general, see \cite{MR2267575, MR2504482}. In this section we
shall use the notation $G$ to denote the group $\glO{n}{2}$, and
the group of invertible matrices of order $n$ over the field $\kk$
is denoted by $\GL_n(\kk)$.

\subsection{Parameterizing set}
Let $\sab \subset \cL^{(2)}$ be the set of tuples $(x_2, x_1)$
satisfying
\begin{enumerate}
\item The module $x_{1}$ has a unique embedding in $x_2$ (up to automorphism).
\item $(x_2,x_1) \le \iota(y)$, $\tau(y)=\la$.
\item $\mathrm{Rank} (x_1) \leq \mathrm{Rank}(x_2/x_1)$.
\end{enumerate}
Let $\pab = \sab/\!\sim$ be the set of equivalence classes in
$\sab$. The uniqueness of embedding implies that $(x_2, x_1)$,
$(y_2, y_1) \in \sab$ are equivalent if and only if $\tau(x_2) =
\tau(y_2)$ and $\tau(x_1) = \tau(y_1)$. Therefore, $\xi = [(x_2,
x_1)] \in \pab$ may be identified with the pair $\mu^{(2)} \supset
\mu^{(1)}$ where $\mu^{(2)} = \tau(x_2)$ and $\mu^{(1)} =
\tau(x_1)$. Further, if $(x_2, x_1) \in \sab$ is such that
$\tau(x_2) = \la$ and $x_1 = \pi x_2$, then the equivalence class
of $(x_2, x_1)$ in $\pab$ is also denoted by $\iota(\la)$. For
$\xi \in \pab$, let
\[
Y_{\xi} = \{ x \in \sab \mid [x] = \xi \}.
\]

Then $\sab = \sqcup_{\xi \in \pab} Y_{\xi}$. Let $F_{\xi} =
\mathbb Q (Y_{\xi})$ be the space of rational valued functions on
$\y{\xi}$. As discussed earlier, the space $\cF_{\iota(\la)}$
coincides with $\cF_{\la}$. We shall prove that $\pab$
parameterizes the irreducible representations of the space
$\cF_{\la}$ and in particular satisfies a relation similar to
(\ref{geometric.RSK}) (See Proposition~\ref{RSK in modules}).

\subsection{An analogue of the RSK correspondence} For
$a \in \mathfrak{o}$ and an $\mathfrak{o}$-module $x$, let $x[a]$
and $ax$ denote the kernel and the image, respectively, of the
endomorphism of $x$ obtained by multiplication by $a$. For any $x
= (x_2, x_1) \in \sab$, the flag of $\pi$-torsion points of $x$,
denoted $x_{\pi}$, is the flag $\kk^n \supseteq x_2[\pi] \supset
x_1 \supset \pi x_2$. In general this flag may not be associated
with a partition but rather a composition. If $x, y \in \sab$ are
such that $[x] = [y]$ then the compositions associated with the
flags $x_{\pi}$ and $y_{\pi}$ are equal. Hence if $[x] = \xi$,
then the composition associated with $x_{\pi}$ is denoted by
$c(\xi)$.
\begin{lem}
\label{lem:reducing cosets to field}
There exists a canonical bijection between the sets \[
\{[(x_2, x_1),(y_2, y_1)] \in Y_{\iota(\la)} \times_{\gg} Y_{\xi} \mid  x_2 \cap y_2 \cong \kk^t, t \in \mathbb N \} \leftrightarrow
X_{c(\iota(\la) )} \times_{\gh} X_{c(\xi)}
\]
obtained by mapping $[(x_2, x_1),(y_2, y_1)] $ to $[(x_2, x_1)_{\pi}, (y_2, y_1)_{\pi}]$.
\end{lem}
\begin{proof}
      Since all the pairwise intersections obtained from the modules $x_2$, $x_1$, $y_2$ and $y_1$ are $\kk$-vector spaces,
      by taking the flags of the $\pi$-torsion points we obtain a well-defined map from $Y_{\iota(\la)} \times_{\gg}
      Y_{\xi}$ to $X_{c(\iota(\la))} \times_{\gh} X_{c(\xi)}$.

      We first prove that this map is injective.
Let $(x,y),(x',y')\in Y_{\iota(\la)} \times Y_{\xi}$ be such that
$x_\pi=x'_\pi$ and $y_\pi=y'_\pi$. Assume that
$[(x_\pi,y_\pi)]=[(x'_\pi,y'_\pi)]$. This means that there exists
an isomorphism $h:\pi\mathfrak{o}_2^n \to \pi\mathfrak{o}_2^n$
such that $h(x_\pi) = x'_\pi$ and $h(y_\pi) = y'_\pi$. We need to
extend $h$ to a map $\tilde{h}:\mathfrak{o}_2^n \to
\mathfrak{o}_2^n$ such that $\tilde{h}(x)=x'$ and
$\tilde{h}(y)=y'$. The elements $x$ and $y$ are tuples of the form
$(x_2,x_1)$ and $(y_2,y_1)$, respectively. We choose maximal free
$\mathfrak{o}_2$-submodules $x_3,x'_3,y_3,y'_3$ of
$x_2,x'_2,y_2,y'_2$, respectively. Since $n \ge 2 (l(\la))$, we
can extend the map $h$ to maps $(x_3+\pi\mathfrak{o}_2^n) \to
(x'_3+\pi\mathfrak{o}_2^n)$ and $(y_3+\pi\mathfrak{o}_2^n) \to
(y'_3+\pi\mathfrak{o}_2^n)$ in a compatible manner such that these
two extensions glue to a unique well-defined map
$(x_3+y_3+\pi\mathfrak{o}_2^n) \to
(x'_3+y'_3+\pi\mathfrak{o}_2^n)$. The latter can now be extended
to an isomorphism $\tilde{h}:\mathfrak{o}_2^n \to
\mathfrak{o}_2^n$ with the desired properties.

To prove surjectivity we need to find a pair $(x,y) \in
Y_{\iota(\la)} \times Y_{\xi}$ which maps to a given pair $(u,v)
\in X_{c(\iota(\la) )} \times X_{c(\xi)}$. This follows at once
from the assumption $n \ge 2 (l(\la))$.

\end{proof}


Let $x = (\kk^n = x_t \supset \cdots \supset x_1 \supset x_0 =
(0))$ be a flag of $\kk$-vector spaces and $v$ be a $k$-vector
space such that $x_i\supset v$ for all $i$, then $x$ modulo $v$,
denoted as $x/v$ is the flag $x/v = (\kk^n /v = x_t/v \supset
\cdots \supset x_1/v \supset (0))$. Let $x=(x_2, x_1), y=(y_2,
y_1) \in \sab$ be such that $x \geq y$. Flags of our primal
interest are $x_{\pi}/\pi y_2$ and $y_{\pi}/\pi y_2$. Observe that
although the flag $y_{\pi}/\pi y_2$ is associated with a
partition, the flag $x_{\pi}/\pi y_2$ may only be associated to a
composition. We say that $y$ embeds into $x$ permissibly if $y
\leq x$ and $y_{\pi}/\pi y_2$ embeds permissibly into $x_{\pi}/\pi
y_2$ (see Section~\ref{subsec: general flags}). For $\eta \leq
\xi$, let $ (Y_{\eta} \times_{\gg} Y_{\xi})^{\circ}$ denote the
set of equivalence classes $[(y,x)] \in Y_{\eta} \times_{\gg}
Y_{\xi}$ such that $y$ embeds permissibly in $x$.
\begin{prop}
\label{RSK in modules}There exists a bijection between the following sets
\[
Y_{\iota(\la)} \times_{\gg} Y_{\xi} \longleftrightarrow
 \bigsqcup_{\eta \in \pab,\eta \leq \xi} (Y_{\eta} \times_{\gg} Y_{\iota(\la)})^{\circ} \times (Y_{\eta} \times_{\gg} Y_{\xi})^{\circ}.
 \]
\end{prop}
\begin{proof}
Let $(x_2, x_1) \in Y_{\iota(\la)}$ and $(y_2, y_1) \in Y_{\xi}$
be elements such that $x_2 \cap y_2 = z_2 \oplus z_1$ such that
$z_2 \cong \fo_2^s $ and $z_1 \cong \fo_1^t$, and denote
$\Omega=[(x_2, x_1), (y_2, y_1)]$. Let $x'_2 = x_2/ z_2$, $y'_2 =
y_2/ z_2$, $x'_1 = x_1/\pi(z_2)$ and $y'_1 = y_1/\pi(z_2)$ then
$x'_2 \cap y'_2 \cong \fo_1^t$. By Lemma~\ref{lem:reducing cosets
to field} and the RSK correspondence, we obtain that the double
coset $[(x'_2, x'_1), (y'_2, y'_1)]$ corresponds to a
$\delta$-flag $(z'_2, z'_1)$ for some partition $\delta$ of $n-s$
with its permissible embeddings $p_1$ and $p_2$ into the flags
$(x'_2, x'_1)_{\pi}/\pi(z'_2)$ and $(y'_2, y'_1)_{\pi}/\pi(z'_2)$
respectively. By adjoining it with $(\fo_2^s, \pi \fo_2^s)$, we
obtain $(u_2, u_1) = (\fo_2^s \oplus z_2, \pi \fo_2^s \oplus z_1)
\in \sab$ with permissible embeddings $p_1$ and $p_2$ in $(x_2,
x_1)$ and $(y_2, y_1)$ respectively. The converse implication
follows by combining the RSK correspondence with the definition of
permissible embedding.
\end{proof}
\begin{remark}
\label{omega notation} Observe that if $\Omega = [(x_2, x_1),
(y_2, y_1)] \in Y_{(\iota(\la))} \times_G Y_{\xi}$ corresponds to
permissible embeddings $p_1$, $p_2$ of $(z_2, z_1)$ in $(x_2,
x_1)$ and $(y_2, y_1)$, respectively, then $\pi(z_2) \cong \pi(x_2
\cap y_2)$. If $[(z_2, z_1)] = (\nu^{(2)}, \nu^{(1)})$, we shall
use the notation $\Omega_{\pi \nu}$ instead of $\Omega$ to specify
this information.
\end{remark}

\subsection{Geometric bases of modules}

The modules $\Hom_{\gg}(\cF_{\xi}, \cF_{\iota(\la)})$ for $\xi \in
\pab$, and in particular the Hecke algebras $\hab =
\mathrm{End}_G(\fab)$, have natural geometric bases indexed by
$\yab \times_{\gg} \y{\xi}$, the space of $\gg$ orbits in $\yab
\times \y{\xi}$ with respect to diagonal action of $\gg$.
Specifically, let
\begin{equation}\label{geometric.basis.mod}
 \gb_\Omega f (x)= \sum_{y:(x,y) \in \Omega} f(y), \qquad f \in
 \f{\xi},\, x \in \yab,
\end{equation}
Then $\{ \gb_{\Omega} \mid \Omega \in \yab \times_{\gg} \y{\xi} \}$ is a basis of $\Hom_{\gg}(\fab, \f{\xi})$.

\subsection{Cellular basis of the Hecke algebras } In this section we determine the cellular basis of the Hecke algebras $\hab$.
Let $\mathcal {R}$ be a refinement of the partial order on $\sab$
given by: For any $(x_2, x_1)$, $(y_2, y_1) \in \sab$, $(x_2, x_1)
\geq_{\cR} (y_2, y_1)$ if either $(x_2, x_1) \geq (y_2, y_1)$ or
$\pi x_2 > \pi y_2$. The set $\pab$ inherits this partial order as
well and is denoted by $\pab^{\cR}$ when considered as partially
ordered set under $\cR$. For $\eta \in \pab$ and orbits $\Omega_1
\in (\y{\eta} \times_{\gg} \y{\xi})^{\circ}$, $\Omega_2 \in
(\y{\eta} \times_{\gg} \yab)^{\circ}$ define
\[
\cb_{\Omega_1 \Omega_2}^{\eta}:= \gb_{\Omega_1^{\mathrm{op}}} \gb_{\Omega_2}.
\]
Then $\cb_{\Omega_1 \Omega_2}^{\eta} \in \mathrm{Hom}_{\gg}(\fab, \f{\xi})$.
Let
\[
\mathcal {C}_{\iota(\la) \xi} = \{ \cb_{\Omega_1 \Omega_2}^{\eta} \mid \eta \in \pab, \eta \leq \xi, \Omega_1 \in (\y{\eta} \times_{\gg} \yab)^{\circ} , \Omega_2 \in (\y{\eta} \times_{\gg} \y{\xi})^{\circ} \}.
\]
\begin{prop}
\label{module basis}
The set $\mathcal {C}_{\iota(\la) \xi}$ is a $\mathbb Q$-basis of the Hecke module $\mathrm{Hom}_{\gg}(\fab, \f{\xi})$.
\end{prop}
\begin{proof} We shall prove this proposition by proving that the transition matrix between the set $\cC_{\iota(\la)\xi}$ and the geometric basis
$\{\gb_{\Omega} \}$ is upper block diagonal matrix with invertible
blocks on the diagonal. Wherever required we also use the notation
$\Omega_{\pi \nu}$ in place of $\Omega$ (see Remark~\ref{omega
notation}). We claim that
\begin{equation}
\label{cell in module} \cb_{\Omega_1 \Omega_2}^{\eta} = \sum_{
\{\Delta_{\pi \chi} \in \yab \times_{\gg} \y{\xi} ~\mid~ \chi
\geq_{\mathcal R} \eta \}} a_{\Delta_{\pi \chi}} \gb_{\Delta_{ \pi
\chi}}.
\end{equation}
Let $[(x,y)] = [(x_2, x_1), (y_2, y_1)] = \Delta_{\pi \chi}$.
Indeed, from the definition of $\cb^{\eta}_{\Omega_1\Omega_2}$ and
$\gb_{\Delta_{\pi \chi}}$, it is clear that the coefficient
$a_{\Delta_{\pi \chi}}$ is given by
\[
a_{\Delta_{\pi \chi}} = |\{ z' \in \sab \mid [z'] = \eta, [(z',x)] = \Omega_1, [(z',y)] = \Omega_2 \}|.
\]
Note that if $(z_2, z_1) \in \sab$ has permissible
embedding in $(x_2, x_1)$ and $(y_2, y_1)$ then $\pi z_2$ embeds
into $\pi (x_2 \cap y_2)$. For the case $\pi z_2 \cong
\pi (x_2 \cap y_2)$, we claim that the coefficients
$a_{\Delta_{\pi \chi}}$ are nonzero only if $\chi \geq \eta$.
Let $z_\pi/\pi z_2$ be a
$\delta$-flag and $\bar{\Omega}_1, \bar{\Omega}_2$ correspond to
permissible embeddings of $z_\pi /\pi z_2$ in $x_\pi/\pi z_2$ and
$y_\pi/ \pi z_2$, respectively. Assume that $x_\pi/\pi z_2$ is a $c_1$-flag and $y_\pi/\pi z_2$ is a $c_2$-flag for some compositions $c_1$ and $c_2$. If $\bar{\Delta}_{\chi} =
[(x_\pi/ \pi z_2, y_\pi/ \pi z_2)]$, then the coefficient of
$\gb_{\bar{\Delta}_{\pi \chi}}$ in the expression of $
c_{\bar{\Omega}_1 \bar{\Omega}_2}^{\gamma} \in \mathcal C_{c_2c_1}$ is
given by
\[
\bar{a}_{\Delta_{\pi \chi}} = |\{ z' \in \cF_{\gamma} \mid [(z' , x_\pi/ \pi z_2)] = \bar{\Omega}_1, [(z', y_\pi/ \pi z_2)] = \bar{\Omega}_2 \}|.
\]
Since by definitions $a_{\Delta_{\pi \chi}}$ = $\bar{a}_{\Delta_{\pi \chi}}$, the coefficient $a_{\Delta_{\pi \chi}}$ is non-zero only if
$\chi \geq \eta$. This implies $\chi \geq_{\cR} \eta$ and completes the proof of (\ref{cell in module}).

By the discussion above we also obtain that by arranging the
elements $\cb_{\Omega_1 \Omega_2}^{\eta}$ and
$\gb_{\Delta_{\eta}}$ for $\eta \in \pab$ in the relation $\cR$,
the obtained transition matrix between the set $$\{\cb_{\Omega_1
\Omega_2}^{\eta} \mid \eta \in \pab, \Omega_1 \in (\y{\eta}
\times_G \y{\xi})^{\circ}, \Omega_2 \in (\y{\eta} \times_G
\yab)^{\circ} \}$$ and  $\{\gb_{\Omega} \mid  \Omega \in \yab
\times_G \y{\xi}\}$ is an upper block diagonal matrix with
invertible diagonal blocks. Observe that the diagonal blocks are
obtained as the transition matrix of certain cellular basis of
Hecke algebras corresponding to the space of flags of $\kk$-vector
spaces to the corresponding geometric basis. This implies that the
set $\mathcal{C}_{\iota(\la) \xi}$ is a $\mathbb Q$-basis of
$\mathrm{Hom}_G(\fab, \f{\xi})$.
\end{proof}
The operation $(\cb_{\Omega_1 \Omega_2}^{\eta})^{\star} =
\cb_{\Omega_2 \Omega_1}^{\eta}$ gives an anti-automorphism of
$\hab$. The $\bQ$-basis of the modules
$\mathrm{Hom}_{\gg}(\f{\xi}, \fab)$ and $\mathrm{Hom}_{\gg}(\fab,
\f{\xi})$ is given by Proposition~\ref{module basis}. This,
combined with the arguments given in the proof of
Theorem~\ref{thm: cellular} proves
\begin{thm} The set $(\mathcal {C}_{\iota(\la) \iota(\la)}, \pab^{\cR})$ is a cellular basis of the Hecke algebra
$\hab$.
\end{thm}

\begin{cor}
\label{main-modules} There exists a collection $\{\cV_{\eta} \mid \eta \in \pab \}$ of inequivalent irreducible representations of
$\glO{n}{2}$ such that
$$ \fab = \oplus_{\eta \in \pab} \cV_{\eta}^{m_{\eta}},$$
where $m_{\eta} = |(\y{\eta} \times_G \yab)^{\circ}|$ is the multiplicity of $\cV_{\eta}$.
\end{cor}


\end{document}